\documentclass[10pt]{article}
\usepackage{amsmath, amsfonts, amssymb, graphicx,psfrag,float,amsthm}
\usepackage{hyperref}
\usepackage{color}
\restylefloat{figure}
\usepackage{sectsty}

\sectionfont{\fontsize{12}{15}\selectfont}

\newtheorem{theorem}{Theorem}
\newtheorem{proposition}{Proposition}
\newtheorem{lemma}{Lemma}

\newtheorem{corollary}{Corollary}

\newtheorem{definition}{Definition}
\newtheorem{algorithm}{Algorithm}
\newtheorem{question}{Question}

\begin{document}

\title{Crosscap Number and Epimorphisms of Two-Bridge Knot Groups}
\author{Jim Hoste\\Pitzer College\\ \\ Patrick D.~Shanahan\\Loyola Marymount University\\ \\Cornelia A.~Van Cott\\University of San Francisco}
\date{}
\maketitle

\begin{abstract}
We consider the relationship between the crosscap number $\gamma$ of knots and a partial order on the set of all prime knots, which is defined as follows. For two knots $K$ and $J$, we say $K \geq J$ if there exists an epimorphism $f:\pi_1(S^3-K) \longrightarrow \pi_1(S^3-J)$. We prove that if $K$ and $J$ are 2-bridge knots and $K> J$, then $\gamma(K) \geq 3\gamma(J) -4$. We also
classify all pairs $(K,J)$ for which the inequality is sharp. A similar result relating the genera of two knots has been proven by Suzuki and Tran. Namely, if $K$ and $J$ are 2-bridge knots and $K >J$, then $g(K) \geq 3 g(J)-1$, where $g(K)$ denotes the genus of the knot $K$. 
\end{abstract}

\vspace{.25cm}

\par{\centering Dedicated to our friend and colleague, Mark Kidwell, 1948--2019.\par}

\section{Introduction}
The {\it crosscap} number, or {\it nonorientable genus},  of a knot $K$ in the 3-sphere, denoted here as $\gamma(K)$,  was introduced by Clark in \cite{Clark_1978}. It is defined as the smallest  first Betti number of any embedded, compact, connected,  nonorientable surface in $S^3$ that spans $K$. For convenience, the crosscap number of the unknot is defined to be $0$. Clark observed that a knot has crosscap number equal to $1$ if and only if it bounds a M\"obius band and hence is a $(2, 2k+1)$-cable of some knot (the centerline of the M\"obius band). Thus the $(2, 2k+1)$-torus knot has crosscap number equal to $1$ for all $k$. 

The crosscap number is not additive with respect to connected sum, however, Clark showed that for knots $K_1$ and $K_2$
\begin{equation}\label{connected sum bounds}
\gamma(K_1)+\gamma(K_2)-1 \le \gamma(K_1 \sharp K_2) \le \gamma(K_1)+\gamma(K_2).
\end{equation}
Clark also showed, that for any knot $K$
\begin{equation}\label{genus bound}
\gamma(K)\le 2 g(K)+1,
\end{equation}
where $g(K)$ is the genus of $K$. Clark asked if equality ever occurs in \eqref{genus bound}, and in \cite{Murakami_Yasuhara_1995}, Murakami and Yasuhara showed that this is the case for the  2-bridge knot $7_4$. This required them to compute $\gamma(7_4)=3$, a computation that has now been made {\sl much} easier by the work of Hatcher-Thurston \cite{HT:1985},  Bessho~\cite{Bessho_1994}, and Hirasawa-Teragaito~\cite{Hirasawa_Teragaito:2006}, who, taken together, provide a practical way to compute the crosscap number of any 2-bridge knot. In particular, one starts by using the work of Hatcher-Thurston to find all compressible, boundary incompressible surfaces that span the 2-bridge knot $K$. We now consider only those surfaces of minimum first Betti number, which we call $n$. Because of the relationship (which is  explained later in this paper) between 2-bridge knots, rational numbers, and continued fractions, it follows that $n$ is the length of the shortest continued fraction associated to $K$. If one of these surfaces is non-orientable, then the crosscap number of $K$ is $n$. If not, then it follows from \cite{Bessho_1994} that the crosscap number of $K$ is $n+1$. Starting with any continued fraction associated to $K$, Hirasawa and Teragaito determine a set of three ``moves'' that can be repeatedly applied to arrive at a shortest continued fraction for $K$, thus determining $n$. Morever, the shortest continued fraction also determines if the crosscap number is $n$ or $n+1$. 

The crosscap number has also been computed for any alternating knot or link from a theoretical standpoint in \cite{Adams-Kindred}. Then, the specific case of alternating knots was addressed concretely and independently in \cite{Ito} and \cite{Kindred_2019}. In addition, crosscap number has been computed for any torus knot~\cite{Teragaito_2004} and any pretzel knot~\cite{Ichihara:2010}. For alternating links, $K$, lower and upper bounds for $\gamma(K)$ that depend on the Jones polynomial of $K$ are given in \cite{Kalfagianni_Lee:2016}. For alternating knots with 12 or fewer crossings, these bounds are often sufficient to determine $\gamma(K)$. 

Building on the results of \cite{Hirasawa_Teragaito:2006}  together with \cite{Agol_2002}, \cite{ALSS_2020}, and \cite{ORS:2008}, we prove the following theorem.

\begin{theorem}\label{main theorem} If $K$ and $J$ are 2-bridge knots and $K>J$, then
$$\gamma(K)\ge 3 \gamma(J)-4.$$
\end{theorem}

Here, by $K\ge J$, we mean the partial order on prime knots defined by $K \ge J$ if there exists an epimorphism $f:\pi_1(S^3-K) \to \pi_1(S^3-J).$ It is clear that this relation is reflexive and transitive. Proving it is antisymmetric is nontrivial. Suppose that $f: \pi_1(S^3-K)\to  \pi_1(S^3-J)$ and $g: \pi_1(S^3-J)\to  \pi_1(S^3-K)$ are epimorphisms. Then, the composition $g \circ f$ is an isomorphism because knot groups are Hopfian (see \cite{Kawauchi}, Lemma~14.2.5). Hence $f$ is an isomorphism and $J=K$ because prime knots are determined by their knot groups \cite{Wh:1987}. It is worth noting that there are other partial orders defined on knots, see  \cite{Endo}, \cite{SW:2006}, \cite{Taniyama}. A closely related order for all knots (not just prime knots)  is given by Silver and Whitten in \cite{SW:2006}. They define $K \ge_p J$ if there is an epimorphism $f:\pi_1(S^3-K) \to \pi_1(S^3-J)$ that preserves peripheral structure. When $K \ge J$ and $K$ and $J$ are 2-bridge, then it follows from work of Agol, Ohtsuki-Riley-Sakuma, and Aimi-Lee-Sakai-Sakuma (summarized in Theorem~\ref{ORS theorem and converse}) that every epimorphism is induced by a branched fold map between the knot exteriors  and, hence, takes meridians to meridians and preserves peripheral structure. Therefore, the two partial orders agree on the set of 2-bridge knots. All of the epimorphisms determined in \cite{HKMS:2011} are meridional, that is,  taking meridians to meridians. In general, however, there are prime knots $J$ and $K$ with $J \ge K$ for which there is no meridional epimorphism \cite{CS:2016}.

A similar result to Theorem~\ref{main theorem} that relates the genera of the two knots has been proven by Suzuki and Tran, \cite{Suzuki_Tran:2018}. Namely, if $K$ and $J$ are 2-bridge knots and $K > J$, then $g(K)\ge 3 g(J)-1$, where $g(K)$ denotes the genus of the knot $K$.

Note that if the hypothesis of Theorem~\ref{main theorem} is relaxed to $K\ge J$, then the result is not true. Because $K\ge K$ for every knot $K$, if the theorem were true, it would imply that $\gamma(K)\ge 3 \gamma(K)-4$, or equivalently, that $2\ge \gamma(K)$ for every knot $K$, which is false. Following the proof of Theorem~\ref{main theorem} we classify all pairs $(K,J)$ for which the inequality is sharp.   There is some evidence that the result is still true even if the assumptions that $K$ and $J$ are 2-bridge knots is dropped. In~\cite{HKMS:2011} and~\cite{Kitano-Suzuki}, the authors determine the partial order denoted by $\ge$ on the set of prime knots with up to 11 crossings. In all cases where $K>J$, the conclusion of Theorem~\ref{main theorem} is true, even when $K$ and $J$ are not both 2-bridge knots.

\begin{question}[]
If $K$ and $J$ are prime but not necessarily 2-bridge knots and $K>J$, then is $\gamma(K)\ge 3 \gamma(J)-4$ still true? What if, 
additionally, $K$ and $J$ are alternating?
\end{question}

In Section~\ref{continued fractions}, we summarize the results of \cite{Hirasawa_Teragaito:2006} that will be needed in this paper. Their analysis depends heavily on representing a 2-bridge knot $K$ by a shortest continued fraction. In Section~\ref{depth}, we recall the definition of the {\it depth} of a rational number defined by \cite{Hirasawa_Teragaito:2006} and use this to show how to derive the crosscap number directly from the unique even continued fraction representing $K$, thus eliminating the need to find a shortest continued fraction for $K$. Morover, as an alternative approach, we  give an algorithm for directly finding a shortest  continued fraction for $K$, thus bypassing the ``shortening moves'' introduced by Hirasawa-Teragaito. We then combine this with the work of \cite{Agol_2002}, \cite{ALSS_2020}, and \cite{ORS:2008} to prove Theorem~\ref{main theorem}. 

\section{Continued Fractions, 2-Bridge Knots, and Crosscap \\Number}\label{continued fractions}

There is a well-established connection between 2-bridge knots, 4-plat diagrams, and continued fractions. We recall  the bare rudiments of this theory and refer the reader to \cite{Cromwell_2004}, \cite{HT:1985},  or \cite{Rolfsen_1990} for more information.
Every 2-bridge knot or link has a 4-plat diagram as in Figure~\ref{4plat}.  Here a box labeled $a_i$ denotes $(-1)^{i+1}a_i$ right-handed half-twists between the two strands entering and exiting the box. Note that a negative number of right-handed twists is equal to the opposite number of left-handed twists and vice-versa. In particular, if $a_i>0$ for all $i$,  or $a_i<0$ for all $i$, then the diagram in Figure~\ref{4plat} is alternating and minimal. That is, it represents a knot with crossing number equal to $|a_1 +a_2 +\dots +a_k|$. See \cite{Cromwell_2004}  for a more extensive description of this topic. The reader is also referred to the paper \cite{Kauffman_Lambropoulou:2003} by Kauffman and Lambropoulou. While not one of the first papers on the subject, their paper is noteworthy for its completeness,  approach, historical description of the topic,  and extensive bibliography.

\begin{figure} 
   \centering
   \includegraphics[width=4in, angle=0]{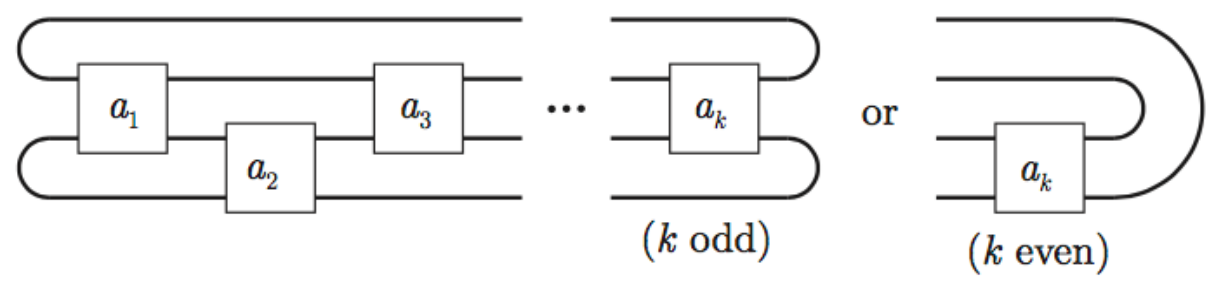}
   \caption{Standard 4-plat diagram of a 2-bridge knot.} 
\label{4plat}
\end{figure}

The sequence $a_1, a_2, \dots, a_k$ gives rise to the reduced fraction $\frac{p}{q}$, via the continued fraction, 
$$\frac{p}{q}=[a_1, a_2, \dots, a_k]=\cfrac{1}{a_1+\cfrac{1}{a_2+\cfrac{1}{\ddots +\cfrac{1}{a_k}}}}.$$ 

This associates to each 4-plat diagram of a 2-bridge knot $K$,  a rational number $p/q$, with $0<p<q$, $q$ odd, and $\gcd(p,q)=1$. (In the case of a 2-bridge link, $p$ will be odd and $q$ even.) Clearly, we can work backwards from any such fraction to obtain a 2-bridge knot, which we denote $K_{p/q}$. Moreover, the fractions $p/q$ and $p'/q'$ represent ambient isotopic 2-bridge knots if and only if $q'=q$ and either  $p' \equiv p \text{ (mod $q$)}$ or $p'p\equiv 1 \text{ (mod $q$)}$. Because we will also consider a 2-bridge knot $K_{p/q}$ and its mirror image, $K_{-p/q}$,  as equivalent, the equivalence on fractions becomes $q'=q$ and either  $p' \equiv \pm p \text{ (mod $q$)}$ or $p'p\equiv \pm1 \text{ (mod $q$)}$.   Note that $\gamma(K_{p/q})=\gamma(K_{-p/q})$.

In Figure~\ref{plumbedBands}, we see that the knot depicted in Figure~\ref{4plat}, can be viewed  as the boundary of a surface obtained by plumbing together $k$ twisted bands, where the $i$-th  band has $a_i$ half twists. 
Hence, each band is either an annulus, if $a_i$ is even, or a M\"obius band, if $a_i$ is odd, and the surface is orientable if and only if each $a_i$ is even. The surface strongly deformation retracts to the cores of the bands and hence  has first Betti number equal to $k$. If $a_i$ is even for all $i$, then the plumbed surface is orientable. However, in this case, Clark~\cite{Clark_1978} credits Mark Kidwell for the observation that introducing a single crossing to the knot diagram by means of a Reidemeister Type I move, corresponds to adding a M\"obius band to the Seifert surface, thus producing a nonorientable spanning surface for the same knot. This is illustrated in the last frame of Figure~\ref{plumbedBands}.
%. 

\begin{figure} 
   \centering
   \includegraphics[width=4in, angle=0]{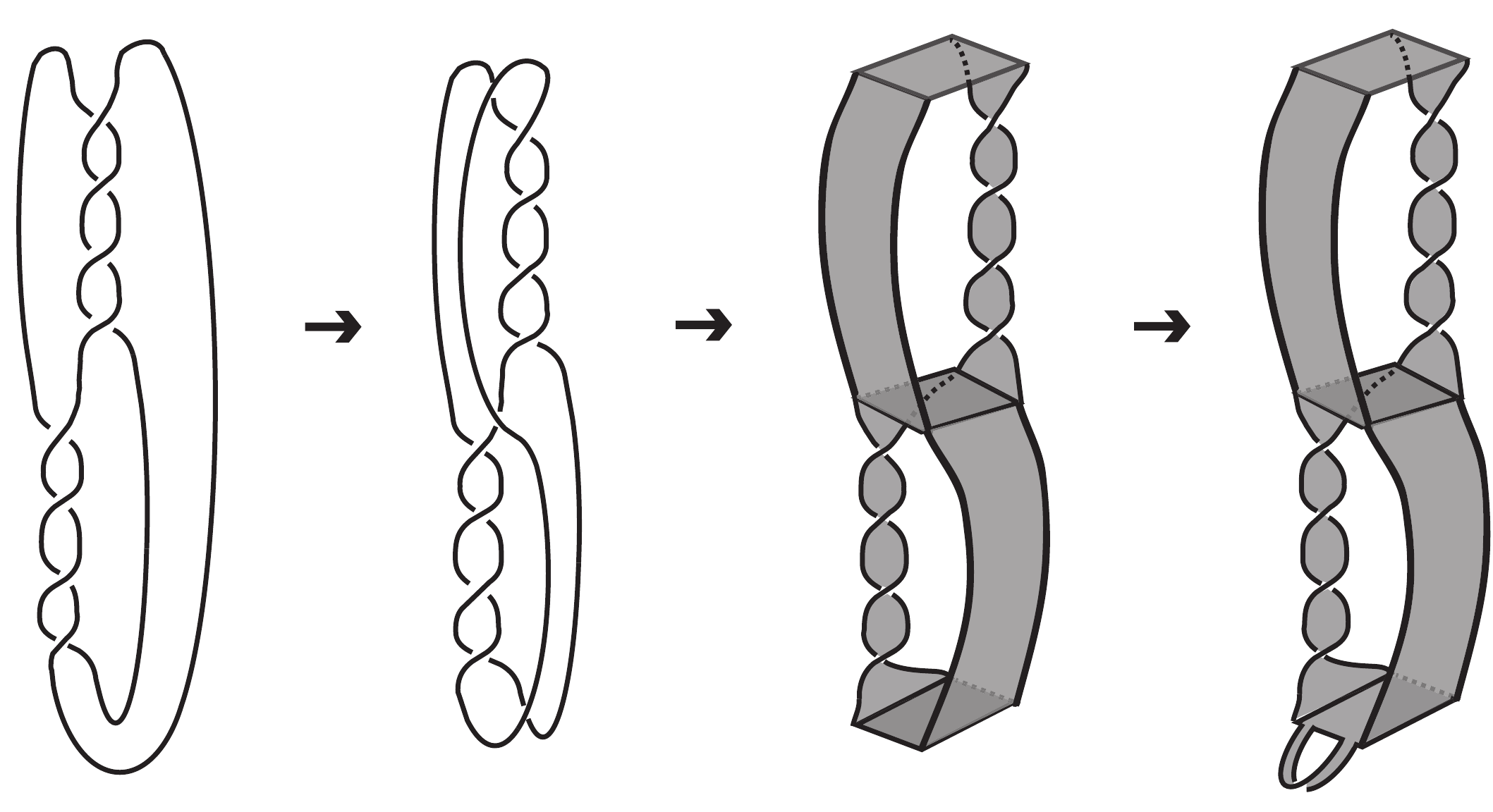}
   \caption{The knot $K_{4/15}$ and its nonorientable spanning surface.} 
%   \caption{A nonorientable spanning surface for $4/15=[4,-4]$ made from two plumbed bands with $a_1=4$ and $a_2=-4$ and then the final addition of a M\"obius band.}
\label{plumbedBands}
\end{figure}

These surfaces are the natural starting point in trying to find a surface with minimal genus (in the orientable case) or minimal first Betti number (in the nonorientable case). As it happens, in both cases, we need look no further. Gabai proved that the minimal genus orientable surface for $K_{p/q}$ is the surface that arises from the continued fraction expansion for $p/q$ which contains all even and nonzero entries~\cite{Gabai_1985}. Similarly, Bessho~\cite{Bessho_1994} proved that the surface given by a shortest--length continued fraction expansion containing an odd integer, realizes the crosscap number. Moreover, if a shortest--length continued fraction produces an orientable surface, then adding a M\"obius band \`a la Kidwell, produces a surface realizing the crosscap number.

The connection between these surfaces associated to $K_{p/q}$ and the length of continued fraction expansions of $p/q$ motivates deeper understanding of continued fraction expansions. Given a fraction $p/q$, there are infinitely many ways to express it as a continued fraction. In an attempt to use only even integers (as mentioned in the result of Gabai), one can obtain  the following result (see Lemma~2.1 in \cite{GHS:2012}).
\begin{lemma} Suppose $0<p<q$, $\gcd(p,q)=1$, $p$ is even and $q$ is odd. Then there exists a unique continued fraction $p/q=[a_1, a_2, \dots, a_{2n}]$ where each $a_i$ is even and nonzero. We call this the {\bf even continued fraction} for $p/q$.
\end{lemma}
Finding this even continued fraction expansion for a given $p/q$ is a straightforward exercise that involves the Euclidean Algorithm (again, see  \cite{GHS:2012} for details).
By comparison, finding a {\em shortest--length} continued fraction expansion associated to $p/q$ is more difficult and such an expansion is not necessarily unique. Hirasawa and Teragaito  \cite{Hirasawa_Teragaito:2006} prove that a continued fraction expansion of $p/q$ can be shortened {\em if and only if} the expansion has one or more of the following three characteristics: 
\begin{itemize}
\item Some $a_i$ in the continued fraction expansion is 0.
\item Some $a_i$ in the continued fraction expansion is $\pm 1$.
\item There exists a substring of $a_i$'s of the form $(2, -2)$ or $(2, -3, 2)$, or\\ $(-2, 3, -3, 2)$, and so on, where the entries in the substring alternate in sign, the substring begins and ends with elements of magnitude 2, and all other elements of the substring, of which there can be any number, including none, have magnitude 3. (We warn the reader  that this notation does not agree with that used by Hirasawa and Teragaito, who use a subtractive form of the continued fraction expression, rather than the additive form that we are using.)
\end{itemize}

Hirasawa and Teragaito define three shortening moves that can be applied to a continued fraction expansion in each of the three cases above. Of the three shortening moves, we will only need the exact description of the first move. The following Lemma is proven in  \cite{Hirasawa_Teragaito:2006}.

\begin{lemma}\label{collapsing zeroes} A zero can be removed from a continued fraction expansion as follows, $$p/q=r+[a_1, a_2, \dots, a_{i-1}, 0, a_{i+1}, a_{i+2}, \dots, a_n]=r+[a_1, a_2, \dots, a_{i-1}+ a_{i+1}, a_{i+2}, \dots, a_n],$$ for all $1<i<n$. (Here $r$ is any integer. We typically omit $r$ if $r=0$.)
\end{lemma}

The main results of \cite{Hirasawa_Teragaito:2006} give the following algorithm for computing $\gamma(K_{p/q})$, which rests on the above discussion.

\begin{algorithm}[Hirasawa--Teragaito]\label{HT algorithm}
To compute the crosscap number $\gamma(K)$ of the 2-bridge knot $K$:
\begin{enumerate}
\item Choose some fraction $p/q$ that represents $K$.
\item Choose some continued fraction expansion $p/q=r+[a_1, a_2, \dots, a_n]$ of length $n$ for $p/q$. This expansion can be shortened if and only if the expansion has one or more of the three characteristics mentioned above.
\item Through repeated application of the three shortening moves, transform the given continued fraction expansion into  a shortest length continued fraction expansion. (We continue here to denote a shortest continued fraction expansion as  $r+[a_1, a_2, \dots, a_n]$.)
\item Once a shortest length continued fraction expansion for $p/q$ has been found, if some $a_i$ is odd or equal to $\pm2$, then $\gamma(K_{p/q})=n$,  otherwise $\gamma(K_{p/q})=n+1$.
\end{enumerate}
\end{algorithm}
As an example, consider the knot $7_4$ which is the 2-bridge knot $K_{4/15}$. Note that $4/15=[4,-4]$. Because none of the entries are $0$, $\pm 1$,  or $\pm 2$, this continued fraction cannot be made shorter using the three shortening moves and hence is shortest. It follows that $\gamma(K_{4/15})=3$. A nonorientable spanning surface with minimal first Betti number is shown in Figure~\ref{plumbedBands}. We invite the reader to compare this with the much more difficult computation given in \cite{Murakami_Yasuhara_1995} which required proving that the crosscap number is not $2$ by first showing that the Goeritz matrix of any knot with crosscap number equal to $2$ has a certain form, and then showing that this form cannot be achieved by the knot $7_4$.

Computing the crosscap number as described in Algorithm~\ref{HT algorithm} has the slight drawback that  no direct instructions are given for finding a shortest--length continued fraction, although it is not hard to imagine an algorithmic scheme for the application of the three shortening moves.  In this paper, we discuss an alternative approach to computing the crosscap number using the unique even continued fraction expansion of $p/q$. This approach  avoids the need to find a shortest-length representative. As a result, we are able to give, later in the paper,  Algorithm~\ref{HSV algorithm} for computing crosscap number that we believe is more direct, simpler,  and faster than Algorithm~\ref{HT algorithm}. Moreover, Algorithm~\ref{HSV algorithm} also facilitates the comparison of the crosscap numbers of 2-bridge knots that are related by the partial order on all prime knots. While somewhat off the main course of this paper, and not needed for the proof of Theorem~\ref{main theorem},  we describe in Algorithm~\ref{grandparentpath} a very direct way to find a shortest continued fraction expansion of $p/q$. However, computer experimentation suggests that this approach is no faster than Algorithm~\ref{HSV algorithm}.

As already mentioned in the Introduction, we may define a partial order on the set of prime knots by declaring $K\ge J$ if there is an epimorphism $\phi:\pi_1(S^3-K) \to \pi_1 (S^3-J)$. Note that every knot is greater than or equal to both itself and the unknot.
In \cite{ORS:2008}, Ohtsuki, Riley and Sakuma systematically construct epimorphisms between 2-bridge knots. 
In particular, they show that if $J$ corresponds to the fraction $p/q=[{\bf a}]$, and  $K$ corresponds to the fraction $p'/q'=[{\bf b}]$,  where the vector {\bf b} has the form
$${\bf b}=(\epsilon_1 {\bf a}, {\bf 2 c_1}, \epsilon_2 {\bf a}^{-1}, {\bf 2  c_2}, \epsilon_3 {\bf a}, {\bf 2  c_3}, \dots, \epsilon_n  {\bf a}^{(-1)^{n-1}}),$$
where each $\epsilon_i \in \{-1, 1\}$, each $c_i$ is an integer,  and if $c_i=0$ then $\epsilon_i=\epsilon_{i+1}$, then there is a {\sl branch fold map} from the complement of $K$ onto the complement of $J$. This map induces a meridional epimorphism from the group of $K$ to the group of the $J$. (Furthermore, the map preserves peripheral structure.) Here we use $({\bf a}, {\bf b})$ to represent the concatenation of the vectors {\bf a} and {\bf b}, denote the reverse of the vector $\bf a$ as ${\bf a}^{-1}$,  and  denote by ${\bf 2  c_i}$ the vector $(2 c_i)$ of length one. Ohtsuki, Riley and Sakuma's result motivates the following definition.

\begin{definition}\label{parsing} {\rm Let ${\bf a}, {\bf b}$ be vectors. We say that $\bf b$ {\bf admits a  parsing with respect to} $\bf a$ if 
$${\bf b}=(\epsilon_1{\bf a}, {\bf2  c_1}, \epsilon_2 {\bf a}^{-1}, {\bf 2  c_2}, \epsilon_3 {\bf a}, {\bf 2  c_3}, \dots, \epsilon_n  {\bf a}^{(-1)^{n-1}}),$$
where each $\epsilon_i \in \{-1, 1\}$, each $c_i$ is an integer, and if $c_i=0$ then $\epsilon_i=\epsilon_{i+1}$.}

\end{definition}
It is easy to prove that if {\bf b} parses with respect to {\bf a}, then this parsing is unique. 
If {\bf b} parses with respect to {\bf a}, then we call the vectors, ${\bf 2  c_i}$, {\bf a}-{\bf connectors}, or more simply, {\bf connectors}. The {\bf a}-connectors separate the  {\bf a}-{\bf tiles}: $\epsilon_1 {\bf a},  \epsilon_2 {\bf a}^{-1}, \dots,$ and $ \epsilon_n  {\bf a}^{(-1)^{n-1}}$. The next theorem forms the basis of our overall approach to the proof of Theorem~\ref{main theorem}.

\begin{theorem}\label{ORS theorem and converse} 
Let $K$ and $K_r$ be 2-bridge knots where  $r = [\mathbf a]=[a_1, a_2, \dots , a_{m}]$ and $|a_i|>1$ for all $i$. Then $K \ge K_r$,  if and only if there exists $\tilde r=[{\bf b}]$ where $\bf b$ parses with respect to $\bf a$ using an odd number of tiles and $K=K_{\tilde r}$. 
\end{theorem}

Theorem~\ref{ORS theorem and converse} follows directly from the work of Ohtsuki-Riley-Sakuma \cite{ORS:2008}, Agol \cite{Agol_2002}, and  Aimi-Lee-Sakai-Sakuma \cite{ALSS_2020}. In particluar, if $\tilde r = [\mathbf b]$ and $\mathbf b$ parses with respect to $\mathbf a$, then  Proposition~5.1 and Theorem~1.1 of \cite{ORS:2008} imply that $K_{\tilde r} \ge K_r$. Conversely, if $K=K_{\tilde r} \ge K_r$ where $r=p/q$, then by Proposition~8.1 of \cite{ALSS_2020} either 
\begin{enumerate}
    \item[(1)] $\tilde r$ or $\tilde r+1$ belongs to the $\hat \Gamma _{p/q}$-orbit of  $p/q$ or $\infty$, or
    \item[(2)]$\tilde r$ or $\tilde r+1$ belongs to the $\hat \Gamma _{p'/q}$-orbit of  $p'/q$ or $\infty$, with $pp'\equiv 1 \text{ (mod $q$).}$
\end{enumerate}
Here $\hat\Gamma_r$ is the group of automorphisms of the Farey graph (defined in Section~\ref{depth}) generated by reflections in edges which have an endpoint equal to either $\infty$ or $r$ and, by Proposition~5.1 of \cite{ORS:2008}, the orbit of $r$ or $\infty$ is precisely those rationals $2c+[\mathbf b]$ where $\mathbf b$ parses with respect to $\mathbf a$. So case (1) gives the desired conclusion. For case (2), by considering equivalent orbits, it is not hard to show that we may assume that $|r'|<1$ which then implies that $p'/q=(-1)^{m-1}[\bf a^{-1}]$ or $p'/q=\pm 1+(-1)^{m-1}[\bf a^{-1}]$. If $p'/q=(-1)^{m-1}[\bf a^{-1}]$, then $\tilde r$ or $\tilde r+1$ corresponds to a vector that parses with respect to $\mathbf a^{-1}$. Finally, if $p'/q=\pm 1+(-1)^{m-1}[\bf a^{-1}]$, then it can be shown (by adjusting connectors appropriately) that there is an $\tilde s$ such that $K_{\tilde s}=K_{\tilde r}$ and $\tilde s$ corresponds to a vector that parses with respect to $\mathbf a^{-1}$.

\section{Depth}\label{depth}

There is a beautiful correspondence between continued fraction expansions of $p/q$ and edge paths from $1/0$ to $p/q$ in the {\it Farey Graph}  shown in Figure~\ref{farey graph}, where the graph is embedded in the Poincar\'e Disk (together with the circle at infinity).  Two fractions, $a/b$ and $c/d$,  in $\mathbb Q \cup \{\frac{1}{0}\}$ are connected by an edge in the Farey Graph if and only if $ad-bc=\pm 1$. The {\bf mediant}, or {\bf child},  of the two fractions $a/b$ and $c/d$ is defined to be $\frac{a+c}{b+d}$ and, therefore, it is connected by an edge to each of its {\bf parents}, $a/b$ and $c/d$. The entire graph can be generated recursively by taking mediants, starting from $1/0$ and $0/1$. (To generate the negative rational numbers, we start with $(-1)/0$ and $0/1$.) Perhaps the best reference for this material is Hatcher's excellent introduction to number theory from a geometric point of view~\cite{Hatcher}.

\begin{figure}[ht] %  figure placement: here, top, bottom, or page
   \centering
   \includegraphics[width=4in]{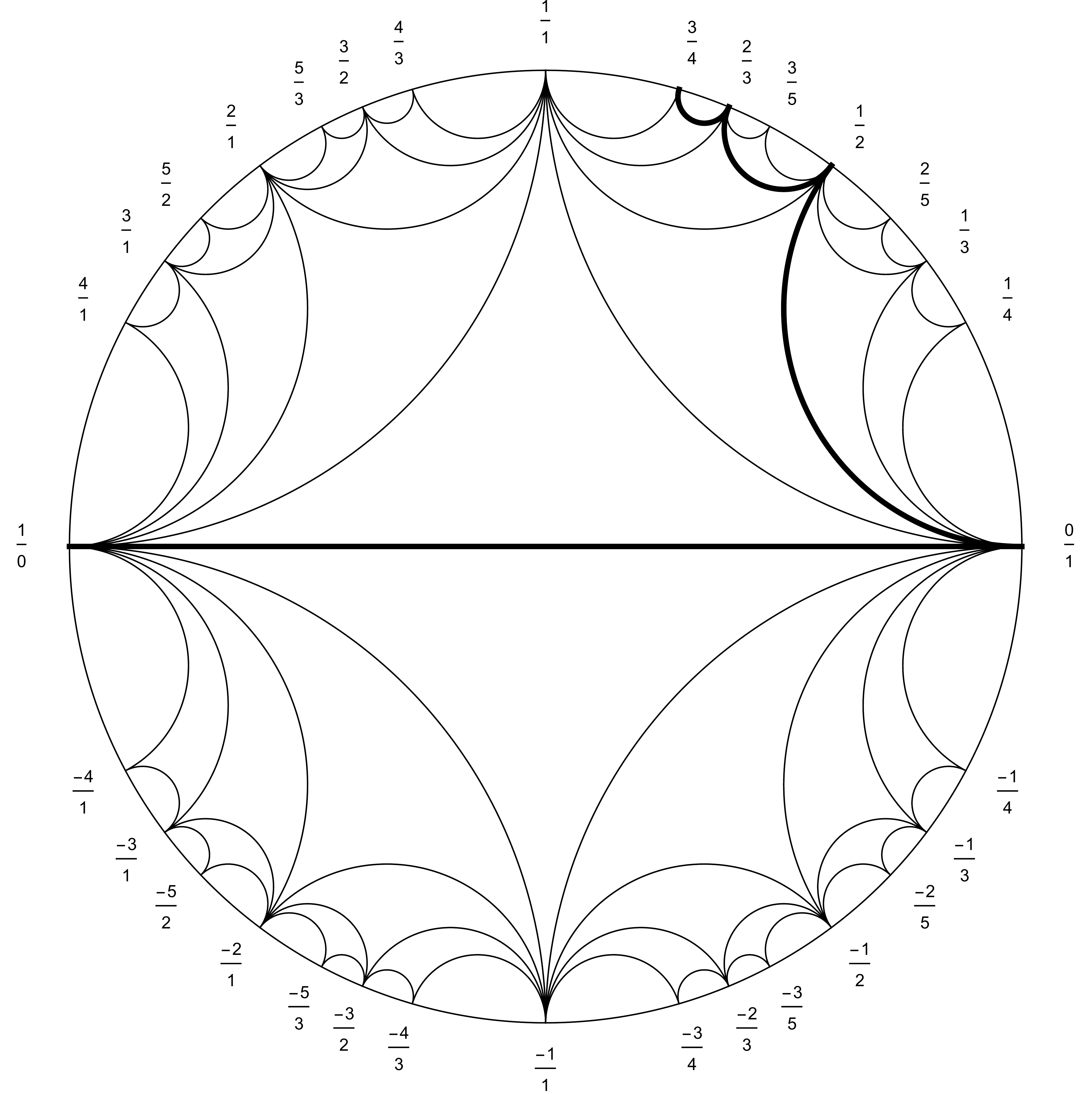} 
   \caption{The Farey Graph  and the path given by $3/4=[2,-2,2]$.}
\label{farey graph}
\end{figure}

Given a continued fraction expansion
$p/q=r+[a_1,a_2, \dots, a_n]$ where $|a_i|>1$  for all $i$, 
let $p_0/q_0=r$ and, for $i>0$, define $p_i/q_i=r+[a_1,a_2, \dots, a_i]$. This defines the {\bf edge path} 
$$\frac{1}{0} \to \frac{p_0}{q_0} \to \frac{p_1}{q_1} \to \dots \to \frac{p_n}{q_n},$$
which always proceeds, repeatedly,  from parent to child, except for the edge $\frac{1}{0} \to \frac{0}{1}$ in the case when $r=0$.
At the vertex $p_{i-1}/q_{i-1}$, the path turns through $|a_i|$ triangles  on the left side of the path if $(-1)^i a_i<0$ and on the right side if $(-1)^i a_i>0$.  Because of this, the entries $a_i$ are sometimes called the {\bf turning numbers}.

For example, the path from $1/0$ to $3/4$ corresponding to the expansion $3/4=[2,-2,2]$ is highlighted in Figure~\ref{farey graph}.  Consecutive edges are always separated by 2 triangles on the left because the turning numbers alternate in sign. In general,  the turning switches from one side of the path to the other when $a_i$ and $a_{i+1}$ have the same sign, and stays on the same side of the path when $a_i$ and $a_{i+1}$ have opposite signs.

Closely related to the length of the shortest path from $1/0$ to the fraction $p/q$ is the depth of $p/q$, which is defined in \cite{Hirasawa_Teragaito:2006} as follows.

\begin{definition}\label{definition of depth} The {\bf depth} of $p/q\in \mathbb Q \cup \{\frac{1}{0}\}$, denoted as $d(p/q)$,  is zero if $p/q\in \mathbb Z \cup \{\frac{1}{0}\}$. Otherwise, the depth of $p/q$ is one more than the minimum of the depths of its two parents.\end{definition} 

It follows directly from the definition that the  depth of a child minus the depth of its parent is always zero or one. The following fact regarding depth is proven in \cite{Hirasawa_Teragaito:2006}.

\begin{lemma}\label{HT depth results} Suppose $p/q=r+[a_1,a_2, \dots, a_n]$ defines the edge path
$$\frac{1}{0} \to \frac{p_0}{q_0} \to \frac{p_1}{q_1} \to \dots \to \frac{p_n}{q_n},$$
from $1/0$ to $p/q$. The path is
a shortest path from $1/0$ to $p/q$ if and only if $d(p_i/q_i)=i$ for all $i\ge 0$. 
\end{lemma}

Given a vector ${\bf a}=(a_1, a_2, \dots, a_m)$, we define its depth as $d({\bf a})=d(p/q)$ where $p/q=[{\bf a}]=[a_1, a_2, \dots, a_m]$.
Once the depth of $p/q$ has been determined, we can compute $\gamma(K_{p/q})$ from the even continued fraction expansion of $p/q$ using the following proposition (which is a restatement of Hirasawa and Teragaito's work).

\begin{proposition}\label{depth or depth plus one}Let $p/q=[{\bf a}]=[a_1, a_2, \dots, a_{2n}]$ be the unique even continued fraction expansion of  $p/q$. Then
$$\gamma \left(K_{p/q} \right) =\left\{ \begin{array}{ll}  d(p/q), & \mbox{if $a_i=\pm 2$ for some $i$,}\\d(p/q)+1, & \mbox{otherwise}. \end{array} \right.$$
\end{proposition}
\begin{proof} If $|a_i| \ge 4$ for all $i$, then none of Hirasawa and Teragaito's three shortening moves apply and hence this continued fraction expansion is shortest and defines a shortest path from $1/0$ to $p/q$. Thus by Lemma~\ref{HT depth results}, it follows that $d(p/q)=2n$. It now follows from Algorithm~\ref{HT algorithm} that $\gamma(K_{p/q})=2n+1=d(p/q)+1$. Otherwise, some $a_i=\pm 2$. If the even continued fraction is shortest, then again,  we obtain $\gamma(K_{p/q})=d(p/q)$.  If the even continued fraction expansion is not shortest, then a shortest one must contain an odd entry because the even continued fraction expansion is unique. If the length of a shortest continued fraction is $m$, then by Algorithm~\ref{HT algorithm} we have that $\gamma(K_{p/q})=m$ and, by Lemma~\ref{HT depth results}, we have that $m=d(p/q)$. Hence $\gamma(K_{p/q})=d(p/q)$.
\end{proof}

The depth of a fraction $p/q$ can be determined from a 2-complex $\mathcal T$ associated to an edge path. Suppose
$$\label{path}\frac{1}{0}\to \frac{0}{1} \to \frac{p_1}{q_1} \to \dots \to \frac{p_n}{q_n}=\frac{p}{q}$$
is the edge path determined by the fraction  $p/q=[a_1, a_2, \dots, a_n]$. As already described, at each vertex the path turns through a number of triangles on either the right or left side of the path. Let $\mathcal T$ be the union of these triangles. We can then determine the depth of each vertex of $\cal T$ recursively from the definition.

For example, consider $10/23=[2, 4, -2, 2]$.  This defines the edge path
$$\frac{1}{0} \to \frac{0}{1} \to \frac{1}{2} \to \frac{4}{9} \to \frac{7}{16} \to \frac{10}{23}.$$
The path and its associated 2-complex are shown in  of Figure~\ref{pathwithdepthandauxdata}. Above or below the fraction at each vertex, we display the depth of the fraction. The depths  are computed recursively using Definition~\ref{definition of depth}. In particular, the depth of $10/23$ is 3.  Notice that the edge path defined by $[2,4,-2,2]$ is not shortest; the depth does not increase by $1$ along the last edge of the path. A shorter path is given by $[2,3,3]$, which starts the same, but then goes from $1/2$ to $3/7$ to $10/23$.

\begin{figure}[htbp]
\begin{center}
   \includegraphics[width=3.0in, angle=90]{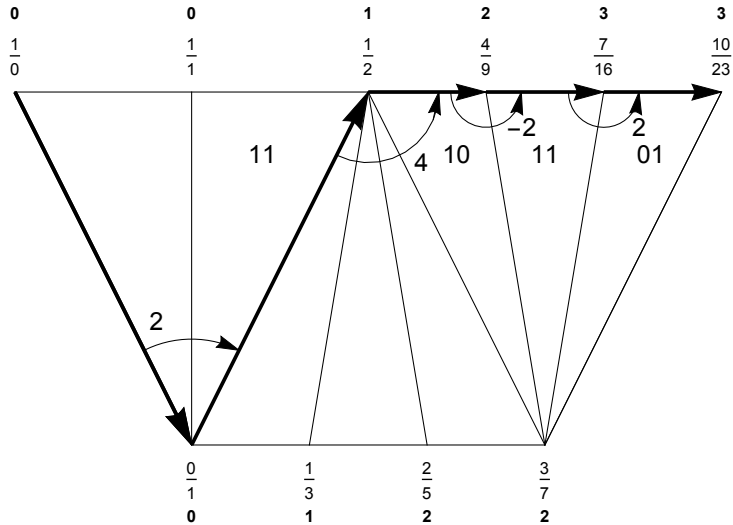} 
\caption{The path associated to $[2, 4, -2, 2]$ with auxiliary data and depths displayed.}
\label{pathwithdepthandauxdata}
\end{center}
\end{figure}

Given $p/q=0+[a_1, a_2, \dots, a_n]$, we will now show that one can compute the depth of  $p/q$ directly from the turning numbers rather than having to draw the 2-complex as in Figure~\ref{pathwithdepthandauxdata} and then compute the depth of every vertex in the 2-complex.
To do this, we introduce the {\bf auxiliary data}, $(d_i, e_i)$ at each turning number $a_i$, as follows. Consider the two edges that terminate at vertex $p_i/q_i$ and which originate at the two parents of this fraction. One of these edges is  in the edge path, but the other is not. Define $d_i$ to be the change in depth along the edge from $p_{i-1}/q_{i-1}$ to $p_i/q_i$ and $e_i$ to be the change in depth along the edge from the other parent. It follows that both $d_i$ and $e_i$ lie in $\{0,1\}$ because each records the change in depth along an edge that points from a parent to a child in the Farey Graph. Thus, the possible values for $(d_i, e_i)$ are, a priori,  limited to $\{(0,0),(0,1),(1,0),(1,1)\}$. However, because our edge paths begin with $1/0\to 0/1$ we will not encounter a triangle with all integral vertices and, therefore, auxiliary data $(0,0)$ will not occur in this context. In Figure~\ref{pathwithdepthandauxdata}, we  have written the pairs $(d_i, e_i)$ in the relevant triangle at each vertex along the path. (In our figures, we have written $11$ for $(1,1)$, etc.).

Clearly, the depth of $p/q$ is the sum $d_1+d_2+\dots+d_n$ because the depth along the path starts at zero and each $d_i$ records the change in depth along one edge in the path.
If $p/q=0+[a_1, a_2, \dots, a_n]$, then we will assume that each $a_i$ is even and nonzero. It is now easy to show that $(d_1, e_1)=(1,1)$ if $|a_1|=2$ and   $(d_1, e_1)=(1,0)$ if $|a_1|\ge 4$. Thus, we can determine the initial auxiliary data from $a_1$. The next result tells us how $(d_i, e_i)$, the sign of $a_ia_{i+1}$,  and $|a_{i+1}|$ determine $(d_{i+1}, e_{i+1})$. This will allow for a recursive calculation of all the auxiliary data, and hence the depth of $p/q$. 

\begin{lemma}\label{aux data depth} Suppose that $p/q=0+[a_1, a_2, \dots, a_{2n}]$ where each $a_i$ is even and nonzero.  If $|a_1|=2$, then  $(d_1, e_1)=(1,1)$; otherwise $(d_1, e_1)=(1,0)$. Then $(d_{i+1}, e_{i+1})$ for $i \ge 1$ is determined recursively by the following table:

\vspace{10pt}
\centerline{\rm
\begin{tabular}{|c|c|c|c|}
\hline
&$|a_{i+1}|=2 \text{ \small \&}$&$|a_{i+1}|=2 \text{ \small \&}$&$|a_{i+1}|\ge 4$\\
$(d_i, e_i)$&$a_ia_{i+1}>0$&$a_ia_{i+1}<0$&\\
\hline
(0,1)&(1,0)&(0,1)&(1,0)\\
(1,0)&(1,1)&(1,1)&(1,0)\\
(1,1)&(1,1)&(0,1)&(1,0)\\
\hline
\end{tabular}
}
\end{lemma}

Before proving this lemma, let us consider an example to illustrate how the lemma works. Consider $92/125=[2, -2, 2, 4, -4, 2]$. According to Lemma~\ref{aux data depth}, $(d_1, e_1) = (1,1)$ because $a_1=2$. Since $a_1a_2=-4$, we look in the last row of the table and the column headed by $|a_{i+1}|=2 $ and $a_ia_{i+1}<0$, to find that $(d_2, e_2)=(0,1)$. Continuing in this way, we see that the  associated sequence of auxiliary data is
$(1,1), (0,1), (0,1), (1,0), (1,0), (1,1)$ 
from which we conclude that the depth of $92/125$ is $1+0+0+1+1+1=4$.

\begin{proof}
We derive the various outcomes in the table for the case where $(d_i,e_i)=(1,1)$ by using Figure~\ref{aux data for 11}. Here we have imagined the last triangle in the 2-complex before advancing one more edge in the path. The last triangle is assumed to have auxiliary data $(1,1)$ and so the depths of its vertices are labeled accordingly, in this case, $d$, $d-1$, and $d-1$. The turning due to the sign of $a_ia_{i+1}$ and the magnitude of $a_{i+1}$, then leads us to the next value of auxiliary data.  Note  that both this figure and its mirror image are needed in general. Similar arguments (and figures) handle the other two cases, where $(d_i,e_i)$ is either $(0,1)$ or $(1,0)$.
\end{proof}
\begin{figure}[ht] 
  \centering
   \includegraphics[width=2.5in,trim=0 100 0 0, clip]{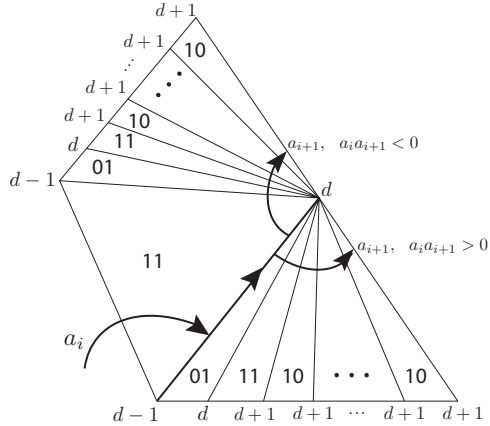} 
   \caption{Determining $(d_{i+1},e_{i+1})$ from $(d_i,e_i)=(1,1)$.}
\label{aux data for 11}
\end{figure}

Considering auxiliary data provides a straightforward proof of the following result which will be used in the proof of Theorem~\ref{main theorem}.\\

\begin{lemma}\label{reversing result} Suppose $\bf a$ is a vector of all even, nonzero  entries. Then
$d({\bf a})=d(- {\bf a})$ and $d({\bf a})=d({\bf a}^{-1}). $
\end{lemma}
\begin{proof} First note that the auxiliary data associated to $-{\bf a}$ is exactly the same as that associated to ${\bf a}$. Thus both have the same depth.
Next, it is well-known that both ${\bf a}$ and ${\bf a}^{-1}$ correspond to the same 2-bridge knot $K$. This is because the 4-plat diagrams determined by ${\bf a}$ and ${\bf a}^{-1}$, respectively, are related by a $180$ degree rotation. Thus both ${\bf a}$ and ${\bf a}^{-1}$ determine knots with the same crosscap number. Because one has an entry of $\pm 2$ if and only if the other one does, Proposition~\ref{depth or depth plus one} implies they must have the same depth.
\end{proof} 

Returning to our definition of the auxiliary data, we introduce two  definitions that are slight variations. Given a vector $\bf a$ all of whose entries are even and nonzero, we can get a different sequence of auxiliary data by still using the table in Lemma~\ref{aux data depth} but by starting with different initial data. Define $d({\bf a}, 01)$ to be the depth derived by starting with initial data $(0,1)$, computing all subsequent auxiliary data according to the table in Lemma~\ref{aux data depth}, and then summing the first entry of each data pair. Similarly, let $d({\bf a}, 10)$ be the analogous depth obtained by using $(1,0)$ as the initial data. 

\begin{lemma}\label{false depths} Suppose $\bf a$ is a vector of all even, nonzero  entries. Then
\begin{itemize}
\item $d({\bf a}, 11)-d({\bf a}, 01)\in \{0,1\}$, 
\item $d({\bf a}, 11)-d({\bf a}, 10)\in \{-1,0\}$, and
\item $d({\bf a}, 10)-d({\bf a}, 01)\in \{0,1,2\}.$
\end{itemize}
\end{lemma}

\begin{proof} First note that the third conclusion follows from the first two because
$$d({\bf a}, 10)-d({\bf a}, 01)=d({\bf a}, 10)-d({\bf a}, 11)+d({\bf a}, 11)-d({\bf a}, 01).$$
To prove the first two assertions, we induct on the length of $\bf a$. If ${\bf a}=(a_1)$, the result is trivial. Now assume that ${\bf a}$ has length bigger than 1 and that the result is true for any shorter vector. Let ${\bf a}=(a_1, a_2, \dots, a_n)$ and ${\bf a}'=(a_2, a_3, \dots, a_n)$. We consider three cases, each corresponding to one of the three (output) columns of Lemma~\ref{aux data depth}.

For the first case, if $|a_2|=2$ and $a_1 a_2>0$, then by considering the three possible initial values for the auxiliary data $(d_1,e_1)$ we obtain that \\
$d({\bf a}, 11) =1+d({\bf a}', 11)$, $d({\bf a}, 10)=1+d({\bf a}', 11)$, and
$d({\bf a}, 01)=0+d({\bf a}', 10)$. Forming the appropriate differences and using the inductive hypothesis immediately gives the result. The proofs of the remaining two cases are similar.
\end{proof}

Returning to Algorithm~\ref{HT algorithm}, we now offer, in passing, a more direct method of finding a shortest continued fraction expansion for any fraction $p/q$. Given $p/q$, first note that it is straightforward to find its two parents in the Farey graph. To do this, first use the Euclidean Algorithm to express the greatest common divisor of $p$ and $q$ as $\pm 1=a q-b p$. Next, let $c=p-a$ and $d=q-b$. Thus the mediant of $a/b$ and $c/d$ is equal to $p/q$. Moreover, $ad-bc=a(q-b)-b(p-a)=aq-bp=\pm 1$. Thus, the fractions $a/b$ and $c/d$ are connected by an edge in the Farey graph. The idea now, starting with $p/q$, is to repeatedly move to the parent with smallest denominator.
This provides a simple algortihm to find a shortest path from $1/0$ to $p/q$.

\begin{algorithm}[]
\label{grandparentpath}
To find a shortest path in the Farey graph from $p/q$ to $1/0$:
\begin{enumerate}
\item If $\frac{p}{q}=\frac{n}{1}$, then a shortest path is $\frac{n}{1}\rightarrow \frac{1}{0}$.

    \item If $\frac{p}{q}= \frac{2k+1}{2}$, then a shortest path is $\frac{p}{q} \rightarrow \frac{k}{1} \rightarrow \frac{1}{0}$.
    \item If $q > 2$, then choose the parent with smallest denominator as the next vertex and repeat until we reach a denominator of $1$ or $2$. Then finish as in (1) or (2), respectively.
\end{enumerate}
\end{algorithm}

\noindent Lemma~\ref{HT depth results} (which is proven in \cite{Hirasawa_Teragaito:2006}) can be used to show that this generates a shortest path from $1/0$ to $p/q$. We leave the details to the reader. 

The following lemma will be used in the proof of Theorem~\ref{main theorem}.

\begin{lemma}\label{isolation lemma} Suppose ${\bf a}$ and ${\bf b}$ are both vectors of all even, nonzero entries and that $c$ is an  integer with $|2c| \ge4$. If ${\bf A}=({\bf a}, 2c, {\bf b})$, then
$$d({\bf A}) = 1+  d({\bf a})+d({\bf b}).$$
\end{lemma}
\begin{proof} First note that the vector $\bf a$ contributes $d({\bf a})$ to the depth of $\bf A$. Let $(d_i,e_i)$ be the auxiliary data of $\bf A$ at $2c$. Because $|2c| \ge 4$, it follows from Lemma~\ref{aux data depth}, that $(d_i ,e_i) =(1,0)$.  Thus, the entry $2c$ contributes $+1$ to the depth of $\bf A$. We also have  that 
$$ (d_{i+1},e_{i+1})=\left \{\begin{array}{cc} (1,1)& \text{if $|b_{1}|=2$,}\\ (1,0) & \text{if $|b_{1}|\ge 4$.}\end{array}\right. 
$$
Thus, the first element $b_1$ of $\bf b$ has the same auxiliary data as it would if only its depth were being computed. So $\bf b$ contributes $d({\bf b})$ to the depth of $\bf A$. Therefore, 
$d({\bf A}) = 1+  d({\bf a})+d({\bf b})$.
\end{proof}

The last result needed for Theorem~\ref{main theorem} is the following.

\begin{lemma} \label{depth of adding one more tile}Suppose ${\bf a}$ and ${\bf b}$ are both vectors of all even, nonzero entries and that $c$ is any integer. Let ${\bf A}=({\bf a},  2c, {\bf b})$, where we further assume that if $c=0$, then the last entry of ${\bf a}$ is equal to the first entry of ${\bf b}$. Then
$$d({\bf A})\ge d({\bf a})+d({\bf b})-1.$$
\end{lemma}
\begin{proof} We break the proof into multiple cases.

\noindent{\bf Case I:} $c=0$. \\
Let ${\bf A}'$ be obtained from ${\bf A}$ by removing $2c$ as in Lemma~\ref{collapsing zeroes}. If 
${\bf a}=(a_1,a_2,\dots, a_m)$ and ${\bf b}=(a_m, b_2, \dots, b_n)$, then ${\bf A}'=(a_1,a_2, \dots, a_{m-1}, 2a_m, b_2,\dots, b_n)$. We also have $d({\bf A})=d({\bf A}')$ since these vectors determine the same fraction. Now $|2a_m|\ge 4$ because $|a_m|\ge 2$.  Therefore,  by Lemma \ref{isolation lemma}, we have
$$d({\bf A}) = d({\bf A}') = 1 + d((a_1,a_2,\dots, a_{m-1})) + d((b_2,\dots, b_{n-1}, b_n)).$$
Removing the last entry from $\bf a$ reduces its depth by at most $1$, and Lemma~\ref{reversing result} implies that removing the first entry from $\bf b$ also reduces its depth by at most 1. Therefore
$d({\bf A}) \ge d({\bf a})+d({\bf b})-1.$\\

\noindent{\bf Case II:} $|2c|\ge 4$. \\
Lemma~\ref{isolation lemma} implies that $d({\bf A}) = 1+d({\bf a})+d({\bf b}) \ge d({\bf a})+d({\bf b})-1$.\\

\noindent{\bf Case III:} $2c=\pm 2$. \\
The sequence of auxiliary data for $\bf A$ begins with the sequence of auxiliary data for $\bf a$ then has the auxiliary data for the entry $2c$, which we denote $(d',e')$, and then has auxiliary data some $(i,j) \in \{(0,1), (1,0), (1,1)\}$ for the first entry $b_1$ of $\bf b$. If we were to compute $d({\bf b})$ then the initial auxiliary data for $b_1$ should be $(1,1)$ if $|b_1|=2$ or $(1,0)$ if $|b_1|\ge 4$. First consider the case where $|b_1|=2$, Now, 
\begin{align*}
d({\bf A})&=d({\bf a})+d'+d({\bf b}, ij)\\
&=d({\bf a})+d'+d({\bf b})+d({\bf b}, ij)-d({\bf b}, 11).
\end{align*}
If $ij=10$, then by Lemma~\ref{false depths}, the difference of the last two terms is either $-1$ or $0$, and because $d'$ is equal to either $0$ or $1$, we obtain the desired result. Similar arguments apply if  $ij \neq 10$  or $|b_1| 
\ge 4$. 
\end{proof}

Lemmas~\ref{aux data depth}--\ref{depth of adding one more tile} allow us to compare the depth of a vector $\bf a$ to the depth of a vector $\bf b$ that parses with respect to $\bf a$.  For example, consider ${\bf a} = [2,2,-2,2]$ which has auxiliary data $(1,1), (1,1), (0,1), (0,1)$.  Note that $d({\bf a}) = d({\bf a}^{-1})=2$ but that the auxiliary data for ${\bf a}^{-1}$ is $(1,1), (0,1), (0,1), (1,0)$. Different types of connectors in a vector that parses with respect to $\bf a$ change the auxiliary data and depth in ways reflected by the lemmas. For example, ${\bf b} = ({\bf a}, -2, {\bf a}^{-1})$ has auxiliary data $(1,1),(1,1),(0,1),(0,1),(0,1),(0,1),(0,1),(0,1),(1,0)$. This data starts with the auxiliary data for $\bf a$ but ends with data for ${\bf a}^{-1}$ that assumes initial data of $(0,1)$ (thus contributing $d({\bf a}^{-1},01)$ to the computation of depth). Since $d({\bf b})= 3 = d({\bf a}) + d({\bf a}^{-1}) -1$, this example illustrates that the inequality in Lemma~\ref{depth of adding one more tile} is sharp.

We are now prepared to prove Theorem~\ref{main theorem}.

\setcounter{theorem}{0}
\begin{theorem}
Suppose $K$ and $J$ are 2-bridge knots with $K>J$. Then $$\gamma(K)\ge 3 \gamma(J) -4.$$ 
\end{theorem}

\begin{proof} We may assume that $J$ corresponds to the fraction $p/q$ with $0<p<q$, $\gcd(p,q)=1$, $p$ is even, and $q$ is odd. Let $p/q=[{\bf a}]=[a_1, a_2, \dots, a_{2n}]$ be the unique even continued fraction expansion of $p/q$. Because $K>J$, it follows from Theorem~\ref{ORS theorem and converse} that $K$ corresponds to a fraction $p'/q'$ with continued fraction expansion given by 
$$p'/q'=[{\bf b}]=[\epsilon_1{\bf a}, 2c_1, \epsilon_2{\bf a}^{-1}, 2c_2, \dots, 2c_{2k}, \epsilon_{2k+1}{\bf a}],$$ 
where each  $\epsilon_i \in \{-1,1\}$, each $c_i$ is an integer, and if $c_i=0$ then $\epsilon_i=\epsilon_{i+1}$. Moreover, $2k+1\ge 3$ because $K \neq J$.
If some connectors, $2c_i$, are equal to zero, then we will collapse them by Lemma~\ref{collapsing zeroes}, to arrive at a shorter vector ${\bf b}'$, which will then have all even and nonzero entries and hence will be the unique even continued fraction expansion for $p'/q'=[{\bf b}']$. As both vectors determine the same fraction, $d({\bf b})=d({\bf b}')$.

We first note that if $\bf a$ contains an entry equal to $\pm 2$, then so does ${\bf b}'$. Hence if $\gamma(J)=d({\bf a})$, then $\gamma(K)=d({\bf b}')$. To see this, assume that
$\bf a$ contains an entry that is $\pm 2$. If this entry is not the first or last entry of $\bf a$, then it appears somewhere in $\bf b$ and continues to appear in ${\bf b}'$ even after any zero-connectors are eliminated. If instead, the first or last entry of $\bf a$ is $\pm 2$, then so is either the first or last entry of $\bf b$, respectively,  and hence also, ${\bf b}'$. Thus ${\bf b}'$ contains an entry that is $\pm 2$.  The second part of the conclusion follows from Proposition~\ref{depth or depth plus one}.

We now consider two major cases.

\noindent{\bf Case I:} The vector $\bf a$ contains some entry equal to $\pm 2$.

By our previous observation, we have that $\gamma(J)=d({\bf a})$ and $\gamma(K)=d({\bf b})=d({\bf b}') $. We can build $\bf b$ in $2k$ steps by starting with $\bf a$ and then repeatedly appending $( 2c_i, \epsilon_{i+1}{\bf a}^{\pm 1})$. Applying Lemma~\ref{depth of adding one more tile} a total of $2k$ times, and using Lemma~\ref{reversing result}, we obtain $d({\bf b})\ge (2k+1)d({\bf a})-2k$. It then follows that
\vspace{-5pt}
\begin{align*}
\gamma(K) &\ge (2k+1) \gamma(J) - 2k\\
&= (2k+1)(\gamma(J)-1)+1 \\
&\ge 3 \gamma(J)-2
\end{align*}

\noindent since $2k+1 \ge 3$.

\noindent{\bf Case II:} No entry of the vector $\bf a$ is equal to $\pm 2$.

Because every entry of $\bf a$ has magnitude 4 or more, every entry has auxiliary data of $(1,0)$. Therefore, 
$d({\bf a})=|{\bf a}|$, where $|{\bf a}|$ is the length of the vector $\bf a$, and moreover $\gamma(J)=d({\bf a})+1= |{\bf a}|+1$. We now break this case into two subcases.

\noindent{\bf Case IIa:} No entry of the vector $\bf a$ is equal to $\pm 2$ but some connector in $\bf b$ is equal to $\pm 2$.

Because some connector in $\bf b$ is $\pm 2$, after collapsing any zero-connectors, it is still the case that ${\bf b}'$ contains an entry equal to $\pm 2$. Thus  $\gamma(K)=d({\bf b}')$. We also have that every entry of  ${\bf b}'$, except possibly some connectors, has magnitude at least 4. Thus, all non-connector entries of ${\bf b}'$ have auxiliary data of $(1,0)$. Because any connector of magnitude 2 follows  an element of magnitude of 4 or more, it has auxiliary data of $(1,1)$. Therefore,  $d({\bf b}')=|{\bf b}'|$. If $z$ is the number of zero-connectors in $\bf b$, then we are assuming that $z\le 2k-1$ and hence that $-2z\ge -4k+2$. Since $2k+1 \ge 3$, we now have
\begin{align*}
\gamma(K)&=|{\bf b}'|\\
&=|{\bf b}|-2z\\
&=(2k+1)|{\bf a}|+2k-2z\\
&\ge(2k+1)|{\bf a}|-2k+2\\
&= (2k+1)(\gamma(J)-2)+3\\
 &\geq 3\gamma(J) - 3
\end{align*}

\noindent{\bf Case IIb:}  No entry of the vector $\bf a$ is equal to $\pm 2$ and no connector in $\bf b$ is equal to $\pm 2$.

In this case, the vector $\bf b$ does not contain an entry equal to $\pm 2$ and this remains true after collapsing any connectors which are equal to zero. Thus  $\gamma(J)=d({\bf a})+1$ and $\gamma(K)=d({\bf b}')+1$. Every entry of both $\bf a$ and ${\bf b}'$ have magnitude $4$ or more. Therefore, we have $d({\bf a})=|{\bf a}|$ and $d({\bf b}')=|{\bf b}'|$. Assuming that $\bf b$ has $z$ zero-connectors with $z\le 2k$, we have
\begin{align*}
\gamma(K)&=|{\bf b}'|+1\\
&=|{\bf b}|-2z+1\\
&=(2k+1)|{\bf a}|+2k-2z+1\\
&\ge(2k+1)|{\bf a}|-2k+1\\
&=(2k+1)(\gamma(J)-2)+2\\
&\ge 3 \gamma(J)-4
\end{align*}
\end{proof}

We point out that in the argument above, the inequality $\gamma(K)\ge 3 \gamma(J) -4$ is strict in both Case I and Case IIa. As a result, Case IIb is the only case where it is possible to achieve equality. Moreover, tracing through the argument in Case IIb, we obtain equality precisely when both $2k+1 = 3$ and every connector is zero. As a result, we have the following corollary.

\begin{corollary}
Suppose $K$ and $J$ are 2-bridge knots with $K>J$. If $\gamma(K)= 3 \gamma(J) -4,$ then $J$ corresponds to an even continued fraction $[{\bf a}] = [a_1, a_2, \ldots, a_{2n}]$ where each $|a_i|\geq 4$, and $K$ corresponds to the continued fraction $[{\bf a}, 0, {\bf a^{-1}}, 0, {\bf a}]$.
\end{corollary}

To illustrate the corollary, let $J_n$ be the 2-bridge knot corresponding to the fraction $\frac{2n}{4n^2+1}=[2n,2n]$ for $n>1$, and let $K_n$ be the 2-bridge knot corresponding to $\frac{32 n^3+6n}{64n^4+20n^2+1}=[2n,2n,0,2n,2n,0,2n,2n]=[2n,4n,4n,2n]$.  This gives infinitely many different pairs of knots, $(K_n, J_n)$, with $K_n>J_n$ by Theorem~\ref{ORS theorem and converse}.  Propositions~\ref{depth or depth plus one} and \ref{aux data depth} tell us that $\gamma(K_n)=5$ and $\gamma(J_n)=3$. Because $5=3\cdot 3-4$,  we see that the inequality of Theorem~\ref{main theorem} is sharp.

Finally, we summarize our work in the following algorithm to compute the crosscap number of any 2-bridge knot $K$. 

\begin{algorithm}[]\label{HSV algorithm}
To compute the crosscap number $\gamma(K_{p/q})$ of the 2-bridge knot $K_{p/q}$, where  $0<p<q$, $\gcd(p,q)=1$,  $p$ is even and $q$ is odd.
\begin{enumerate}
\item Compute the unique even continued fraction expansion $p/q=[a_1, a_2, \dots, a_{2n}]$ where each $a_i$ is even and nonzero.
\item Compute the auxiliary data and $d(p/q)$ using  Lemma~\ref{aux data depth}.
\item If $a_i =\pm 2$ for some $i$, then $\gamma(K) = d(p/q)$; otherwise $\gamma(K) = d(p/q)+1$.
\end{enumerate}
\end{algorithm}

A Jupyter notebook containing a Python program that implements Algorithm~\ref{HSV algorithm} is included with the arXiv upload of this paper.

\begin{figure}
\begin{center}
\begin{tabular}{cc|rrrrrrrr|r}
\multicolumn{9}{c}{\hspace{1cm}Crosscap number, $\gamma(K)$}\\
& &1&2&3&4&5&6&7&8&Total\\
\hline
 &3 & 1 & 0 & 0 & 0 & 0 & 0 & 0 & 0 & 1 \\
 &4 & 0 & 1 & 0 & 0 & 0 & 0 & 0 & 0 & 1 \\
 &5 & 1 & 1 & 0 & 0 & 0 & 0 & 0 & 0 & 2 \\
 &6 & 0 & 2 & 1 & 0 & 0 & 0 & 0 & 0 & 3 \\
 &7 & 1 & 2 & 4 & 0 & 0 & 0 & 0 & 0 & 7  \\
 &&&&&&&&&&\vspace{-2cm}  \\
 {\rotatebox[origin=c]{90}{Crossing number, cr(K)}}&  8 & 0 & 3 & 7 & 2 & 0 & 0 & 0 & 0 & 12 \\
 &&&&&&&&&&\vspace{-2cm}  \\
 &9 & 1 & 3 & 12 & 8 & 0 & 0 & 0 & 0 & 24 \\
 &10 & 0 & 4 & 17 & 21 & 3 & 0 & 0 & 0 & 45 \\
 &11 & 1 & 4 & 26 & 43 & 17 & 0 & 0 & 0 & 91 \\
 &12 & 0 & 5 & 33 & 78 & 53 & 7 & 0 & 0 & 176 \\
 &13 & 1 & 5 & 44 & 127 & 136 & 39 & 0 & 0 & 352 \\
 &14 & 0 & 6 & 53 & 194 & 278 & 150 & 12 & 0 & 693 \\
 &15 & 1 & 6 & 68 & 280 & 526 & 419 & 87 & 0 & 1387 \\
 &16 & 0 & 7 & 79 & 389 & 889 & 989 & 375 & 24 & 2752 \\
 \end{tabular}
\end{center}
\caption{Crosscap numbers by crossing number for 2-bridge knots with 16 or fewer crossings.}
\label{distribution}
\end{figure}

\setcounter{section}{3}
\section{Computations}

Using our implementation of Algorithm~\ref{HSV algorithm}, we computed the crosscap number of every 2-bridge knot with 16 or fewer crossings. Our computations agree with the data in \cite{knotinfo}, for knots up to 12 crossings (the limit of KnotInfo's table).  A table of crosscap numbers is given in \cite{Hirasawa_Teragaito:2006} for 2-bridge knots with 12 or fewer crossings. Our computations agree with their data as well. In Figure~\ref{distribution}, we give the distribution of crosscap numbers for all 2-bridge knots up to 16 crossings. It was proven in \cite{Murakami_Yasuhara_1995} that for any knot $K$
$$\gamma(K)\le \Bigl \lfloor \frac{cr(K)}{2} \Bigr \rfloor, $$
where $cr(K)$ is the crossing number of $K$ and $\lfloor x \rfloor$ is the floor of $x$, that is, the largest integer less than or equal to $x$. This inequality is confirmed by our computations and is borne out in Table~\ref{distribution}.

\bibliographystyle{plain}
\bibliography{myReferences}

\end{document}